\documentclass[reqno,12pt]{amsart}

\usepackage[utf8]{inputenc}

\usepackage{amssymb,amsmath,enumitem,amsthm,mathrsfs} 

\newtheorem{defin}{Definition}
\newtheorem{thm}[defin]{Theorem}

\newtheorem{lemma}[defin]{Lemma}
\newtheorem{cor}[defin]{Corollary}

\newtheorem{oq}[defin]{Open question}

\newcommand{\IOpen}{\mathrm{IOpen}}
\newcommand{\modm}{\mathcal{M}}

\newcommand{\IP}{\subseteq^{IP}}
\newcommand{\RC}[1]{\text{R}(#1)}
\newcommand{\FF}[1]{\text{F}(#1)}

\newcommand{\Pres}{\mathrm{Pr}}

\newcommand{\IOpenLin}{\mathrm{IOpenLin}}

\newcommand{\Ilx}{I^l_x}
\newcommand{\eIOpen}{\mathrm{eIOpen}}

\newcommand{\str}[1]{\StrLen{#1}[\tmplength]\ifthenelse{1=\tmplength{}}{\mathcal{#1}}{\underline{#1}}}

\newcommand{\M}{\mathcal{M}}

\newcommand{\stru}[2]{\lara{#1,#2}} 
\newcommand{\lang}[1]{\lara{#1}} 
\newcommand{\thn}[1]{\mathrm{#1}} 

\newcommand{\PA}{\thn{PA}}
\newcommand{\PrA}{\thn{Pr}}

\newcommand{\DeLOs}[2]{{\tiny\ifthenelse{\equal{#1}{#2}}{\thn{DeLO}^{#1}}{\ifthenelse{\equal{#2}{}}{\thn{DeLO}^{#1}}{\thn{DeL0}^{\vpair{$#1$}{$#2$}}}}}}

\newcommand{\lekv}{\leftrightarrow}
\newcommand{\limp}{\rightarrow}
\renewcommand{\land}{\,\&\,}

\newcommand{\N}{\mathbb{N}}

\newcommand{\lr}[1]{\{#1\}}

\newcommand{\lara}[1]{\langle #1\rangle}

\newcommand{\set}[2]{\lr{#1;#2}}

\newcommand{\vect}[1]{\overline{#1}}

\newcommand{\uvz}[1]{``#1''} 

\newcommand{\ipi}[1]{\mathrm{IPR}(#1)}
\usepackage{color}

\begin{document}

\title{Shepherdson's theorems for fragments of open induction}

\abstract{By a well-known result of Shepherdson, models of the theory $\mathrm{IOpen}$ (a first order arithmetic containing the scheme of induction for all quantifier free formulas) are exactly all the discretely ordered semirings that are integer parts of their real closures. In this paper we prove several analogous results that provide algebraic equivalents to various fragments of $\mathrm{IOpen}$.
}


   
\author{Jana Glivick\' a}
\address{Jana Glivick\' a: Department of Theoretical Computer Science and Mathematical Logic, Faculty of Mathematics and Physics, Charles University, Malostranské náměstí~25, 118~00 Praha~1, Czech Republic}
          
\author{Petr Glivick\' y}
\address{Petr Glivick\' y: Department of Mathematics, Faculty of Informatics and Statistics, University of Economics, Prague, Ekonomická~957, 148~00 Praha~4, Czech Republic}
}

\thanks{This paper was processed with contribution of long term institutional support of research activities by Faculty of Informatics and Statistics, University of Economics, Prague. The study was supported by the Charles University, project GA UK No.~270815. The work was supported by the grant {SVV-2016-260336}.}

\maketitle

%

\section{Introduction}

In \cite{Shep} Shepherdson proved that there is a recursive nonstandard model of the open induction arithmetic $\mathrm{IOpen}$ (in contrast to Peano arithmetic ($\PA$), where no such model exists by the Tennenbaum's theorem \cite{Ten59}). Shepherdson's model is constructed as an integer part of certain real closed field (see the Preliminaries section for the precise definitions). Implicitly, even more is proved: a discretely ordered semiring $\M$ is a model of $\mathrm{IOpen}$ if and only if $\M$ is an integer part of the real closure $\RC{\M}$ of $\M$.

In this paper we prove several analogous results -- versions of the Shepherdson's theorem -- for other fragments of $\PA$ in place of $\mathrm{IOpen}$ that correspond to various algebraic properties of (extensions of) their models.

\section{Preliminaries}

\subsection{Discretely ordered rings and their extensions}
A discretely ordered ring is a structure $\mathcal{R}=\stru{R}{0,1,+,-,\cdot,\leq}$ such that $\stru{R}{0,1,+,\allowbreak -,\cdot}$ is a commutative ring, $\leq$ is a linear ordering on $R$ such that $1$ is the least positive element, and $\leq$ respects $+$ and $\cdot$ in the following way:%
$$a\leq b \limp a + c \leq b + c,\ \ \ \ \ \ 0 \leq a, b \limp 0 \leq a\cdot b,$$
for all $a,b,c\in R$. 

A nonnegative part of a discretely ordered ring in the language without $-$ is called a discretely ordered semiring. We denote the semiring corresponding to the ring $\M$ by $\M^+$.

Further on, $\modm$ always denotes a discretely ordered ring. 
(Such an $\M$ necessarily contains negative elements and therefore, strictly speaking, can not be a model of arithmetical theory $T$. If we say that $\M$ is a model of $T$, which we denote by $\M\models T$, we mean that $\M^+$ is.)

We define $\FF{\modm}$ as the fraction field of $\modm$ and $\RC{\modm}$ as the unique, up to isomorphism (by the Artin-Schreier theorem \cite[Theorem B.14]{Markerbook}), ordered real closure of $\FF{\modm}$ that preserves the ordering of $\FF{\M}$. By $\frac{\M}{\N}$ we denote the ordered ring of all formal fractions of the form $m/n$, where $m\in M$ and $0\neq n\in\N$.

Let $\mathcal{R}$ be a ring. A discretely ordered subring $I$ of $\mathcal{R}$ is called an integer part of $\mathcal{R}$ (denoted by $I \IP \mathcal{R}$) if for every $r \in R$, there is $i \in I$ such that $r-1 < i \leq r$. 

We call $m\in M$ the integer part of $r\in\RC{\M}$ if $r-1 < m \leq r$. (Note that every $r\in\RC{\M}$ has an integer part $m\in M$ if and only if $\M\IP\RC{\M}$.)

We say that an ordered ring $\mathcal{R}'$ is a dense subring of an ordered ring $\mathcal{R}$, and denote it by $\mathcal{R}' \subseteq^d \mathcal{R}$, if $\mathcal{R}'$ is a subring of $\mathcal{R}$ and for every $q<r$ from $R$ there is $r'\in R'$ such that $q<r'<r$.

\subsection{Integer-parts-of-roots property}
Let $f(x)$ be a definable unary function on $\M$. By $\ipi{f}$ (integer-parts-of-roots) we denote the following formula:
$$
(a<b \land f(a)\leq y<f(b))\limp(\exists x)(a \leq x < b \land f(x)\leq y < f(x+1)).
$$
The intended meaning of $\M\models\ipi{f}$ can be expressed in vague terms as \uvz{existence of integer parts for all $f$-roots of values $y\in M$} or, in other words, \uvz{existence of integer parts of all values $f^{-1}(y)$, where $f^{-1}$ is an inverse function of $f$, i.e. a function such that $f(f^{-1}(y))=y$, for all $y\in M$}.

If $\mathcal{F}$ is a set of definable unary functions on $\M$, we write $\ipi{\mathcal{F}}$ for the scheme $\set{\ipi{f}}{f\in\mathcal{F}}$.


\subsection{Arithmetical theories}
\label{sect:arithmeticaltheories}

Now we define several arithmetical theories that we use in this paper. 
Robinson arithmetic ($\mathrm{Q}$) is a basic theory of arithmetic in the language $L=\lang{0,1,+,\cdot,\leq}$. It's axioms are just elementary properties of the symbols from the language. For our purposes the precise axiomatics is not important. We refer the reader to \cite[page 28, Definition 1.1]{HParitmeticbook}.

Peano arithmetic is the extension of $\mathrm{Q}$ by the scheme
\begin{equation}
\label{eq:indukce}
(\varphi(0,\vect{y}) \land (\forall x)(\varphi(x,\vect{y})\limp\varphi(x+1,\vect{y}))\limp (\forall x)\varphi(x,\vect{y})),
\end{equation}
of induction for all $L$-formulas $\varphi(x,\vect{y})$ with a distinguished variable $x$ (see e.g. \cite{Kayearithmeticbook} for a detailed list of axioms and basic properties of $\PA$).

When we extend Robinson arithmetic just by the scheme \eqref{eq:indukce} for all quantifier free (also called open) formulas, we get the arithmetic of open induction ($\IOpen$). Its models are exactly all the discretely ordered semirings that satisfy the induction axioms \eqref{eq:indukce} for all quantifier free formulas $\varphi$.

The theory $\mathrm{IOpenLin}$ (open linear induction) is the extension of Robinson arithmetic by the induction scheme \eqref{eq:indukce} for all linear formulas $\varphi(x,\vect{y})$ with a distinguished variable $x$. Here, we say that $\varphi(x,\vect{y})$ is linear if in every occurrence of $\cdot$ in $\varphi$ at least one of the two factors is $y_i$, for some $i$.

Presburger arithmetic ($\PrA$) is the theory of the structure ${\stru{\N}{0,1,+,\leq}}$. It can be explicitly axiomatized as the theory in the language $L^+=\lang{0,1,+,\leq}$ containing the following axioms:
\begin{center}
\begin{tabular}{cccc}
(A1) & $0\neq z+1$, & (A2) & $x+1=y+1\limp x=y$,\\
(A3) & $x+0=x$, & (A4) & $x+(y+1)=(x+y)+1$,
\end{tabular}\\
\smallskip
(D${}_\leq$) $x\leq y \lekv (\exists z) (x+z=y)$,
\end{center}

and the scheme of induction \eqref{eq:indukce} for all formulas of the language $L^+$.

\section{Results}





We describe how the relations between $\modm$, $\FF{\modm}$ and $\RC{\modm}$ translate to certain forms of induction in $\modm$. 
The Shepherdson's result on the relation between $\IOpen$ and integer parts of real closures \cite{Shep} can be reformulated in the following way:

\begin{thm}
\label{thm:Shepreformul}
The following are equivalent for any discretely ordered ring $\modm$:
\begin{enumerate}
	\item $\modm \IP \RC{\modm}$, \label{SH1}
	\item all roots $r\in\RC{\M}$ of polynomials $p\in M[x]$ have integer parts in $M$, \label{SH2}
	\item $\M\models\ipi{M[x]}$, \label{SH3}
	\item $\M^+ \vDash \IOpen$. \label{SH4}
\end{enumerate}
\end{thm}

For Presburger arithmetic a similar theorem easily follows from well known properties of $\PrA$:

\begin{thm}
\label{thm:Pres}
Let $\modm$ be a discretely ordered ring. Then the following are equivalent:
\begin{enumerate}
	\item $\modm \IP \frac{M}{\mathbb{N}}$,
	\item all fractions $m/n$ with $m\in M$ and $0\neq n\in\N$ have integer parts in $M$,
	\item $\M\models\ipi{\set{n(x)}{n\in\N}}$, where $n(x)=x+\ldots+x$ with $n$ summands,
	\item $\modm^+ \vDash \Pres$.
\end{enumerate}
\end{thm}



Theorems \ref{TFAE} and \ref{IEO_Claim} state analogous results for other arithmetical theories, providing equivalents to the following algebraic properties: $\modm \IP \FF{\modm}$ and $\FF{\modm} \subseteq^d \RC{\modm}$.

\begin{thm} \label{TFAE}
Let $\M$ be a discretely ordered ring. Then the following are equivalent:
\begin{enumerate}
	\item $\M \IP \FF{\M}$,\label{IOL1}
	\item all fractions $m'/m$ with $m',m\in M$ and $0<m$ have integer parts in $M$,\label{IOL1.5}
	\item $\M \models \ipi{\set{m(x)}{m\in M}}$, where $m(x)=m\cdot x$,\label{IOL1.75}
	\item[(*3)] $\M \vDash (\forall n,k \neq 0)(\exists n')(\exists 0 \leq l < k)(n = n' \cdot k + l)$,\label{IOL2}
	\item $\M^+ \vDash \IOpenLin$. \label{IOL3}
\end{enumerate}
\end{thm}
\begin{proof}
Clearly, $\ref{IOL1}\Leftrightarrow\ref{IOL1.5}$ and $\ref{IOL1.75}\Leftrightarrow{}^*\ref{IOL2}$. Further:
\begin{itemize}
	\item $\ref{IOL1} \Rightarrow \ref{IOL3}$: Let $\varphi(x,\bar{y})$ be a quantifier free linear formula, $\bar{m} \in M$. We prove that the induction axiom for $\varphi(x,\bar{m})$ holds in $\M$. $\varphi(x,\bar{m})$ is equivalent to a boolean combination of formulas of the form $x \cdot u \geq v$, for some $u,v \in M$. Suppose that $\M^+ \vDash \varphi(0,\bar{m}) \wedge \neg(\forall x) \varphi(x,\bar{m})$. The set $\{x \in M^+; \M^+ \vDash \varphi(x,\bar{m})\}$ is a finite union of intervals with endpoints of the form $\frac{v}{u}$, for some $u,v \in M$, or $\pm \infty$. By condition \ref{IOL1}, for any $u,v \in M$, there exists $w \in M$ such that $w = \left\lfloor \frac{v}{u}\right\rfloor$. This implies that $\M^+ \vDash (\exists w)(\varphi(w,\bar{m}) \wedge \neg \varphi(w+1, \bar{m}))$.
	\item $\ref{IOL3} \Rightarrow {}^*\ref{IOL2}$: Fix $n,k \neq 0 \in M^+$. $\M \vDash 0 \cdot k \leq n \wedge \neg(\forall n')(n' \cdot k \leq n)$, however $n' \cdot k \leq n$ is an open linear formula. Condition \ref{IOL3} implies $\M^+ \vDash (\exists n')(n' \cdot k \leq n < (n' + 1) \cdot k)$; such an $n'$ and $l:=n - n' \cdot k$ satisfy the condition ${}^*$\ref{IOL2}.
	\item ${}^*\ref{IOL2} \Rightarrow \ref{IOL1}$: Let $\frac{n}{k} \in \FF{\M}$ and $n',l \in M$ satisfying the condition ${}^*$\ref{IOL2} for $n,k$. Then $\frac{n}{k} = \frac{n' \cdot k + l}{k} = n' + \frac{l}{k}$ and $\left|n' - \frac{n}{k}\right| = \frac{l}{k} < 1$.
\end{itemize}
\end{proof}

\begin{lemma} \label{<1density}
Let $r \in \RC{\M}$ satisfy $0 < r < 1$. Then there is some $m \in M^+$ such that $0 < \frac{1}{m} < r$.
\end{lemma}
\begin{proof}
Suppose for contradiction that for every $m \in M^+$ it is $0 < r < \frac{1}{m}$. Let $f(x) \in \FF{\M}[x]$ be of the least degree such that $r$ is the root of $f(x)$. Then the absolute coeficient of $f(x)$ is $0$ (if not then it can be expressed as a linear combination of powers of $r$ which is easily a contradiction) and $f(x) = g(x) \cdot x$ for some $g(x) \in \FF{\M}[x]$. This contradicts the minimality of degree of $f(x).$
\end{proof}

\textbf{Notation:}
\\
$$\Ilx \varphi(x,\bar{y}) = \Ilx \varphi := ((\forall u < l) \varphi(u,\bar{y}) \wedge (\forall v)(\varphi(v,\bar{y}) \rightarrow \varphi(v+l, \bar{y}))) \rightarrow (\forall x)\varphi(x,\bar{y})$$

$\eIOpen$ is an axiom schema $(\exists l > 0)(\Ilx \psi(\frac{x}{l},\bar{y}))$, where $\psi(x,\bar{y})$ is an open formula.

If we rewrite $\Ilx \psi(\frac{x}{l},\bar{y})$ we get the following formula
$$ ((\forall u < l) \psi(\frac{u}{l},\bar{y}) \wedge (\forall v)(\psi(\frac{v}{l},\bar{y}) \rightarrow \psi(\frac{v}{l}+1, \bar{y}))) \rightarrow (\forall x)\psi(\frac{x}{l},\bar{y}).$$

\begin{thm} \label{IEO_Claim}
Let $\M$ be a discretely ordered ring. Then the following are equivalent:
\begin{enumerate}
	\item $\FF{\M} \subseteq^d \RC{\M}$, \label{IEO1}
	\item $(\forall r \in \RC{\M})(\exists f \in \FF{\M}) (r < f < r+1)$, \label{IEO2}
	\item $(\forall r \in \RC{\M}) (\exists l,m \in M) (m \leq lr < m + l)$ [i.e. $\frac{m}{l} \leq r < \frac{m}{l} + 1$], \label{IEO3}
	\item $\M^+\vDash \eIOpen$. \label{IEO4}
\end{enumerate}
\end{thm}
\begin{proof}
$\ref{IEO1} \Rightarrow \ref{IEO4}$: Let $\psi(x,\bar{y})$ be an open formula, $\bar{a} \in M^+$; in the next, $\psi(x)$ means $\psi(x,\bar{a})$. $\psi(x)$ can be written as a boolean combination of fomulas of the form $p(x) \geq 0$, where $p(x) \in \FF{\M}[x]$. $\psi(x)$ can change its truth value only at roots of these polynomials. 

Suppose there is some $f \in \FF{\M}$ such that $0 \leq f < 1$ and $\neg \psi(f)$. As $f = \frac{u}{l}$ for some $u < l \in M^+$, we have $\neg \psi(\frac{u}{l})$ and $\Ilx \psi(x)$ holds. 

In the next, we suppose $\psi(x)$ holds on $[0,1) \cap \FF{\M}$. Let $r \in \RC{\M}$ be the largest root of some polynomial from $\psi(x)$ such that $\psi(x)$ holds on $[0,r) \cap \FF{\M}$ (if there is not such an $r$, $\psi(x)$ holds on all $f \in \FF{\M}$; in particular, $\Ilx \psi(x)$ holds for any $l > 0$). Let $r' \in \RC{\M}$ be the smallest root of some polynomial from $\psi(x)$ greater then $r$ (or $r' := +\infty$, if such a root does not exist). By \ref{IEO1}, there exists some $f \in \FF{\M}$ with $1 \leq r < f < r'$, $f < r+1$. Let us take $m,l \in M^+$ such that $f - 1 = \frac{m}{l}$. Then $0 < \frac{m}{l} < r < \frac{m}{l} + 1 < r'$. Then, by the choice of $r, r'$, it holds $\psi(\frac{m}{l})$ and $\neg \psi(\frac{m}{l} + 1)$. That implies $\Ilx \psi(\frac{x}{l})$.

$\ref{IEO4} \Rightarrow \ref{IEO3}$: $r \approx r'$ denotes that there is no $f \in \FF{\M}$ between $r$ and $r'$. Let us fix $r \geq 1$, the case for $0 \leq r < 1$ is trivial, $r < 0$ symmetric. Note that it suffices to find some $m,l \in M$ such that \ref{IEO3} holds for $m,l$ and any $r' \approx r$. Let $p(x) \in \FF{\M}[x]$ be such that $p(r) = 0$. By (repeateadly) differentiating $p(x)$, we arrive at some polynomial $p'(x)$ such that it has only one, simple root $r' \approx r$. We may suppose that $p'(x)$ is increasing at $r'$. Choose $f,f' \in \FF{\M}$ so that $r'$ is the only root of $p'(x)$ between them. 

Consider $\psi(x) := (x \leq f') \wedge (x \leq f \vee p'(x) \leq 0)$. Obviously, $\psi(x)$ holds on $[0, r']$ and does not hold on $(r', +\infty)$. By \ref{IEO4}, there exists $l \in M$, $l > 0$ such that $\Ilx \psi(x)$. We chose $r \geq 1$, thus also $r' \geq 1$. Therefore $\psi(\frac{u}{l})$ holds for every $u < l$. On the other hand, $\neg (\forall x) \psi(\frac{u}{l})$. From $\Ilx \psi(x)$, we get $m \in M$ such that $\psi(\frac{m}{l})$ and $\neg \psi(\frac{m}{l} + 1)$. This implies $\frac{m}{l} \leq r' < \frac{m}{l} + 1$ and we are done.

$\ref{IEO3} \Rightarrow \ref{IEO2}$: Easy.

$\ref{IEO2} \Rightarrow \ref{IEO1}$: Let $r < r' \in \RC{\M}$. By Lemma \ref{<1density}, we may fix $0 < k \in M$ such that $\frac{1}{k} < r' - r$. Then $kr < kr + 1 < kr'$. By \ref{IEO2}, there exists $f \in \FF{\M}$ for which $kr < f < kr + 1 < kr'$ holds. This implies $r < \frac{f}{k} < r'$. Since $\frac{f}{k} \in \FF{\M}$ we are done.
\end{proof}

\medskip

As clearly $(\M\IP \FF{\M} \land \FF{\M}\subseteq^d\RC{\M}) \limp \M\IP\RC{\M}$, we get the following:
\begin{cor}
$\IOpen = \IOpenLin + \eIOpen$.
\end{cor}
\begin{proof}
The inclusion $\subseteq$ follows from the above implication. The other inclusion is trivial as all instances of the induction scheme in $\IOpenLin$ and $\eIOpen$ are just special cases of open induction.
\end{proof}

Similarily from the inclusion $\IOpen \supseteq \IOpenLin + \eIOpen$, we get:
\begin{cor}
$(\M\IP \FF{\M} \land \FF{\M}\subseteq^d\RC{\M}) \lekv \M\IP\RC{\M}$.
\end{cor}

\section{Questions}
The obvious similarity of Theorems \ref{thm:Shepreformul}, \ref{thm:Pres}, \ref{TFAE} (and \ref{IEO_Claim}) suggests that there may be a common generalization. We state this as a rather vague
\begin{oq}
Is there a truly general Shepherdson's theorem, i.e. a theorem such that Theorems \ref{thm:Shepreformul}, \ref{thm:Pres}, \ref{TFAE} (and \ref{IEO_Claim}) are its special cases?
\end{oq}

The theory $\IOpen$ is the strongest fragment of $\PA$ that we have a Shepherdson's type theorem for. Recently, however, Shepherdson's original ideas were generalized by Ko{\l{}}odziejczyk \cite{Kol11} and used to construct interesting models of the theory $T^0_2$ (an arithmetic with sharply bounded induction -- see \cite[Section 2]{Kol11} for the precise definition) that is stronger than $\IOpen$. Therefore it seems interesting to ask:
\begin{oq}
Are there variants of the Shepherdson's theorem for theories stronger than $\IOpen$? In particular is there such a variant for the theory $T^0_2$?
\end{oq}


\bibliography{bibliography}{}
\bibliographystyle{amsalpha}

%

\end{document}